\documentclass[12pt]{amsart}

\usepackage[utf8]{inputenc}
\usepackage{mathrsfs}
\usepackage{amssymb,amsmath,amscd}
\usepackage{setspace}
\usepackage{color}
\usepackage[all]{xy}
\usepackage{graphics}
\usepackage[english]{babel}
\usepackage{url}

\usepackage{hyperref}

\newtheorem{Theorem}{Theorem}[section]
\newtheorem{Lemma}[Theorem]{Lemma}
\newtheorem{Proposition}[Theorem]{Proposition}
\newtheorem{corollary}[Theorem]{Corollary}

\theoremstyle{remark}
\newtheorem{remark}{Remark}
\newtheorem{example}{Example}

\newcommand{\z}{\mathbb Z}

\newcommand{\m}{\mathfrak M}

\newcommand{\g}{\mathcal{G}}

\newcommand{\hp}{\mathbb H \mathrm P}
\newcommand{\op}{\mathbb O \mathrm P}
\newcommand{\Ker}{\mathrm{Ker}}

\title[Mapping class group of projective planes]{The mapping class group of manifolds which are like projective planes}

\author{Yang Su}
\address{School of Mathematics and Systems Science, Chinese Academy of Sciences, Beijing 100190, China}
\address{School of Mathematical Sciences, University of Chinese Academy of Sciences, Beijing 100049, China}
\email{suyang@math.ac.cn}

\author{Wei Wang}
\address{Department of Mathematics and Computational Science, Shanghai Ocean University, Shanghai 201306, China}
\email{weiwang@amss.ac.cn}

\date{}

\begin{document}

\maketitle

\begin{abstract}
In this paper we compute the mapping class group of simply-connected closed smooth manifolds $M$ with integral homology $H_{*}(M) \cong \z \oplus \z \oplus \z$ provided that $\dim M \ne 4$.
\end{abstract}

\section{introduction}
A manifold which is like a projective plane is a simply-connected closed
smooth manifold $M$ such that $H_*(M;\z)\cong \z \oplus \z \oplus \z$. Besides the spheres these are the simply-connected manifolds with the simplest homology group. Solution of the Hopf invariant $1$ problem implies that the dimension of $M$ is necessarily $2m=4$, $8$ or $16$. Typical examples are the projective planes $\mathbb C \mathrm P^2$, $\mathbb H \mathrm P^2$ and $\mathbb O \mathrm P^2$ over complex numbers, quaternions and octonions respectively. In dimension $4$, such a manifold is homeomorphic to $\mathbb C \mathbb P^2$ by Freedman's classification of simply-connected $4$-manifolds. In dimension $8$ and $16$ these manifolds were first studied by Eells and Kuiper in \cite{EellsKuiper62} as manifolds which admit a Morse function with three critical points. In that paper the authors obtained a classification of these manifolds up to connected sum with homotopy spheres. This was completed to a diffeomorphism classification by Kramer and Stolz (\cite[Theorem A and Theorem 1.3]{KramerStolz}).  It turns out that the diffeomorphism class of a projective plane like manifold $M$ of dimension $2m \ge 8$ is determined by its Pontrjagin class $p_{m/4}(M)$. More precisely, let $x \in H^m(M)$ be a generator, then $M=M_t$ has Pontrjagin class 
\[
\small{
p_{m/4}(M_t)= \left \{ \begin{array}{ll} 
2(1+2t)x, \   t \equiv 0,7,48,55 \pmod {56} & m=4 \\
6(1+2t)x, \ t \equiv 0, 127, 16128, 16255 \pmod{16256} & m=8 \end{array} \right.}\]
This implies that  the $\hat A$-genus is a complete numerical invariant of these manifolds
\[
\small{
\hat A(M_t)= \left \{ \begin{array}{lll} 
-\frac{t(t+1)}{7 \cdot 8}, &  \ t \equiv 0,7,48,55 \pmod {56} & m=4 \\
-\frac{t(t+1)}{127\cdot 128}, & \ t \equiv 0, 127, 16128, 16255 \pmod{16256}& m=8 \end{array} \right.}
\]

\medskip

Manifolds which are like projective planes are interesting objects  in Riemannian geometry. The quaternionic and octonionic projective planes are compact symmetric spaces of rank one and endowed with positive sectional curvature metrics. Other manifolds in the class also support interesting Riemannian metrics.  For example, it is shown in \cite{TangZhang} that $\mathbb H \mathrm P^2$-like manifolds, alias \emph{Eells-Kuiper quaternionic projective planes}, admit $SC^p$-Riemannian metrics, meaning that all geodesics starting from a point $p$ are periodic with the same length.

\medskip

The aim of this paper is to study the mapping class group $\m(M)$ of these manifold. This is the group of isotopy classes of orientation-preserving self-diffeomorphisms of $M$, and can be viewed as the  path components of the group of orientation-preserving self-diffeomorphisms $\mathrm{Diff}(M)$, which is a topological group under the $C^{\infty}$-topology. Note that since the Pontrjagin class $p_{m/4}(M)$ is non-trivial and the cohomology ring $H^*(M)$ is isomorphic to the truncated polynomial ring $\mathbb Z[x]/(x^3)$ with deg $x=m$, any self-diffeomorphism of $M$ 
is necessarily
orientation-preserving.  

To state the results we need to introduce some notations. We denote by $\m(D^n,  \partial )$ the mapping class group of the $n$-disc relative to its boundary. When $n\geqslant 5$, this group is isomorphic to the group of $(n+1)$-dimensional homotopy spheres $\Theta_{n+1}$ by the gluing construction (\cite[Corollary 2]{Cerf}). For any $n$-manifold $M$, fix an embedded $n$-disc $D \subset M$,  there is a homomorphism $$\m(D^n, \partial) \to \m(M)$$
given by extending a diffeomorphism on $D$ by the identity on the complement of $D$, such a diffeomorphism is called a \emph{local diffeomorphism}. The $\alpha$-invariant 
$ \hat{\mathscr{A}} \colon \Omega_*^{\mathrm{spin}} \to KO^{-*}(\mathrm{pt})$ is a ring homomorphism generalizing the $\hat A$-genus for $4k$-dimensional spin manifolds  (\cite[\S 4.3]{Hitchin}). The composition
$$\Theta_{2m+1} \to  \Omega_{2m+1}^{\mathrm{spin}} \stackrel{\hat{\mathscr A}}{\longrightarrow} KO^{-(2m+1)}(\mathrm{pt}) =\z/2$$ 
is surjective. An exotic sphere with non-trivial $\alpha$-invariant is usually called a \emph{Hitchin sphere}. 

\begin{Theorem}\label{thm:main1}
Let $M$ be an $\mathbb H \mathrm P^{2}$-like manifold, the mapping class group $\m(M)$ is isomorphic to $\z/2$, generated by any local diffeomorphism whose corresponding exotic sphere has non-trivial $\alpha$-invariant. For those $M$ with $\hat{A}(M)$ even, there is an isomorphism 
$$\m(M) \stackrel{\cong}{\longrightarrow} \z/2, \ \ f \mapsto \hat{\mathscr{A}}(M_f)$$
where $M_f$ is the mapping torus of $f$.
\end{Theorem}

The quaternion projective plane $\mathbb H \mathrm P^2$ admits Riemannian metrics of positive scalar curvature, and the space $\mathcal{R}^+_{\mathrm{scal}}$ of such metrics is shown to be not path-connected \cite[\S 4.4]{Hitchin}.  The group $\mathrm{Diff}(\mathbb H \mathrm P^2)$ acts on $\mathcal{R}^+_{\mathrm{scal}}$ via pulling back metrics by a diffeomorphism. The quotient space $\mathcal M^+_{\mathrm{scal}}$ is the moduli space of Riemannian metrics on $\mathbb H \mathrm P^2$ with positive scalar curvature. The above action descends to an action of the mapping class group on the set of path components of $\mathcal{R}^+_{\mathrm{scal}}$ with quotient the set of path components of $\mathcal M^+_{\mathrm{scal}}$. As a corollary of Theorem \ref{thm:main1} we have

\begin{corollary}\label{cor:curvature}
\begin{enumerate}
\item The map $\pi_0\mathcal{R}^+_{\mathrm{scal}} \to \pi_0 \mathcal M^+_{\mathrm{scal}}$ is a two-to-one correspondence. 
\item For a Riemannian metric on $\mathbb H \mathrm P^2$ of positive scalar curvature any isometry is isotopic to the identity. 
\end{enumerate}
\end{corollary}

\begin{remark}
There are other geometric topological consequences of Theorem \ref{thm:main1}. For example, let $\mathrm{haut}(M)$ be the group of homotopy classes of homotopy equivalences of $M$, then the natural homomorphism 
$$\m(M) \to \mathrm{haut}(M)$$ 
is trivial, since any local diffeomorphism is homotopic to the identity. By \cite[\S 9]{Baues} and \cite[\S 6]{EellsKuiper62}, the group  $\mathrm{haut}(M)$ is isomorphisc to $\z/2$. Therefore there is an ``exotic" self-homotopy equivalence of $M$ which is not homotopic to any diffeomorphism. Theorem \ref{thm:main1} also implies  that the action of $\mathrm{haut}(M)$ on the smooth structure set $\mathscr S^{\mathrm{DIFF}}(M)$ is free.
\end{remark}

To state the results for $\op^2$-like manifolds we recall some facts about the stable homotopy groups of sphere $\pi_*^S$. We follow the notations in \cite[Chap XIV]{Toda}: 
there are elements $\nu \in \pi_3^S$ of order $8$ and $\kappa \in \pi_{14}^S$ of order $2$, their product $\nu  \kappa$ generates a $\z/2$-direct summand of $\pi^S_{17}$; there is an element $\bar \mu \in \pi_{17}^S$ of order $2$, generating another $\z/2$-direct summand of $\pi_{17}^S$.  For more information of $\pi_{17}^S$ see Lemma \ref{kerHurewicz}. There is a split short exact sequence
$$ 0 \to bP_{18} \to \Theta_{17} \to \mathrm{Coker}J \to 0$$
where $J \colon \pi_{17}(O) \to \pi_{17}^S$ is the $J$-homomorphism. 

\begin{Theorem}\label{thm:main2}
Let $M$ be an $\op^2$-like manifold,  then the mapping class group $\m(M)$ is isomorphic to $(\z/2)^3$. More precisely, there is a split central extension
$$ 0 \to K \to \m(M) \to \z/2 \to 0$$
where $K$ is isomorphic to $\z/2 \oplus \z/2$, generated by  local diffeomorphisms whose images are $\bar \mu$ and $\nu \circ \kappa$ respectively.
\end{Theorem}

\medskip

In the remaining part of the introduction we briefly describe the strategy used to determine the mapping class group of $M$.
The embedding theorem of Haefliger and the $h$-cobordism theorem of Smale provide a geometric description of projective-plane like manifolds (c.~f.~\cite[\S1]{KramerStolz}).  Let $M^{2m}$ ($m=4,8$) be a projective plane like manifold, by Haefliger's embedding theorem, there is an embedding $S^m \to M$, unique up to isotopy, representing a generator of $H_m(M;\z) \cong \z$. Let $N$ be a closed tubular neighborhood of $S^m$, which is the disc-bundle of the normal bundle $\xi$ of  $S^m$. The $h$-cobordism theorem implies $M=N \cup_{S^{2m-1}} D^{2m}$.  The mapping class groups of $M$ and $N$ are related by the following exact sequence (\cite[Theorem 3] {Wall2})  
\begin{equation}\label{sequence}
\m(D^{2m}, \partial) \stackrel{\alpha}{\longrightarrow}  \m(M) \stackrel{\beta}{\longrightarrow} \m(N)\stackrel{\gamma}{\longrightarrow} I(M) \rightarrow 0
\end{equation}
Here $\alpha$ is given by local diffeomorphism; by the Disk Theorem, any diffeomorphism $f$ of $M$ is isotopic to a diffeomorphism $f_1$ which fixes $D$,  $\beta(f)$ is the restriction of $f_1$ on $N$; $I(M)=\{ \Sigma \in \Theta_{2m} \ | \ M \sharp \Sigma \cong M \} $ is the inertia group of $M$, and $\gamma$ is the restriction of a diffeomorphism of $N$ on $\partial N = S^{2m-1}$.
It's shown in \cite[Theorem A]{KramerStolz} that 
$$I(M) = \Theta_{2m} \cong  \z/2.$$ 
The group $\m(D^{2m}, \partial)$ is identified by the glueing construction with the group of homotopy $(2m+1)$-spheres $\Theta_{2m+1}$, we denote the homomorphism corresponding to $\alpha$ by $\phi \colon \Theta_{2m+1} \to \m(M)$. Therefore the task here is to understand the homomorphisms $\gamma$ and $\phi$. This will be done in \S \ref{sec:bundle} and \S\ref{sec:local} respectively. Finally in \S \ref{sec:proof} we prove the theorems.

Manifolds which are like projective planes are  $(m-1)$-connected 
$2m$-manifolds. 
A classical computation of the mapping class group of $(m-1)$-connected $2m$-manifolds was given in \cite{Kreck78}, where the manifolds are assumed to be almost parallelizable.  Projective plane-like manifolds have non-trivial Pontrjagin class hence are not almost parallelizable. As an effort to extend Kreck's result to non almost-parallelizable highly-connected manifolds some new ideas and techniques are brought into play in this paper.  

\

\noindent \textbf{Acknowledgement.} The first author is partially supported by NSFC 12071462. 
 
\section{Bundle automorphisms} \label{sec:bundle}
In this section we first determine the group $\m(N)$ and the homomorphism $\gamma$.  Recall that $N$ is a closed disk bundle of $\xi$. 
Let $\mathrm{aut}(\xi)$  be the isotopy classes of orientation-preserving bundle automorphisms of $\xi$. A bundle automorphism clearly induces a diffeomorphism of $N$, we have a well-defined homomorphism $\mathrm{aut}(\xi) \to \m(N)$.  

\begin{Lemma}
The homomorphism $\mathrm{aut}(\xi) \to \m(N)$ an isomorphism.
\end{Lemma}

\begin{proof}
A diffeomorphism $f$ necessarily preserves the Pontrjagin class $p_{m/4}(M)$ hence induces the identity on $H_m(N)$. The zero-section $S^m \subset N$ represents a generator of $H_{m}(N)$, after an isotopy we may assume that $f|S^m=\mathrm{id}$.  By the tubular neighborhood theorem, such a diffeomorphism is isotopic to one induced by a bundle automorphism of $\xi$. Therefore the homomorphism $\mathrm{aut}(\xi) \to \m(N)$ is surjective.  When $m=4$, by the facts that $\mathrm{aut}(\xi) \cong \z/2$ (Lemma \ref{autxi}) and $\gamma \colon \m(N) \to I(M) = \Theta_8 \cong \z/2$ is surjective, $\mathrm{aut}(\xi) \to \m(N)$ must be an isomorphism. 

When $m=8$ we define an inverse homomorphism $\m(N) \to \mathrm{aut}(\xi)$ as follows. For any diffeomorphism $f \colon N \to N$ with $f|_{S^8} = \mathrm{id}$,  the differential $df$ of $f$ induces a bundle automorphism of $\xi$. We need to show that the isotopy class of $df$ depends only on the isotopy class of $f$. Let $Q \colon N \times I \to N \times I$ be an isotopy between $f$ and $g$ with $f|_{S^8}$=$g|_{S^8}=\mathrm{id}$. The induced isomorphism $Q_*$ on $H_9(N \times I, N \times \{0,1\})$ is the identity, therefor $Q|_{S^8 \times I}$ is homotopic to the identity relative to $S^8 \times \{0,1\}$. By Haefliger's embedding theorem \cite[Theorem 1(b)]{Haef},  $Q|_{S^8 \times I}$ is isotopic to the identity relative to $S^8 \times \{0,1\}$. Modifying $Q$ by an isotopy, we may assume that $Q|_{S^8 \times I}=\mathrm{id}$. Then the differential $dQ$ gives an isotopy  between $df$ and $dg$.

The composition $\mathrm{aut}(\xi) \to \m(N) \to \mathrm{aut}(\xi)$ is by definition the identity. The composition $\m(N) \to \mathrm{aut}(\xi) \to \m(N)$ is the identity by \cite[Lemma 9(ii)]{Wall2}.
\end{proof}

Next we compute the group of bundle automorphisms $\mathrm{aut}(\xi)$. From the correspondence between the vector bundle $\xi$ and its associated principal $SO(m)$-bundle $P(\xi)$, it's easy to see that the group of bundle automorphisms of $\xi$ is identified with the group of principal bundle automorphisms of $P(\xi)$, i.~e.~the gauge group $\g (P(\xi))$. Therefore $\mathrm{aut}(\xi) \cong \pi_0(\g(P(\xi))$.

For the convenience of the reader we briefly recall  the homotopy theory of gauge groups, for more details see \cite{AtiyahBott,Gottlieb,Theriault}. For a principal $G$-bundle  $P \to B$, let $\g^P_G(B)$ be its gauge group, then the classifying spaces of the topological group $\g^P_G(B)$ is homotopy equivalent to the connected component $\mathrm{Map}_P(B,BG)$ of the mapping space $\mathrm{Map}(B,BG)$ which contains the classifying map of $P$. The homotopy equivalences 
$$\g^P_G(B)\simeq \Omega B\g^P_G(B)\simeq 
\Omega \mathrm{Map}_P(B,BG)$$
induce an isomorphism $\pi_0 \g^P_G(B)\cong \pi_1 \mathrm{Map}_P(B,BG)$. 
There is a fibration sequence 
$$ \Omega BG  \to \mathrm{Map}^*_P(B,BG) \to \mathrm{Map}_P(B,BG)\stackrel{\mathrm{ev}}{\longrightarrow} BG $$
where $\mathrm{ev}$ is the evaluation map, $\mathrm{Map}^*_P(B,BG)$ is the corresponding space of pointed maps, and $\Omega BG$ is homotopy equiavlent to $G$. We have an associated exact sequence of homotopy groups
$$ \pi_1(G) \stackrel{\partial}{\longrightarrow} \pi_1(\mathrm{Map}^*_P(B,BG)) \to \pi_1(\mathrm{Map}_P(B,BG)) \to \pi_0(G)$$
For $B=S^n$ and $G$ connected, there are homotopy equivalences 
$$\mathrm{Map}^*_P(S^n,BG)\simeq \Omega^{n}(BG)\simeq \Omega^{n-1}G$$ 
and we get an exact sequence
\begin{equation}\label{sequence2}
\pi_1(G)\stackrel{s_P}{\longrightarrow} \pi_n(G) \to \pi_0(\g^P_G(B)) \to 0
\end{equation}
The homomorphism $s_P \colon \pi_1(G) \to \pi_n(G)$ can be described as follows. 
For a topological group $G$, the commutator map $G \times G \to G$, $(g_1, g_2) \mapsto [g_1, g_2]$  descends to  a map
$c \colon G\wedge G\rightarrow G$, which induces a bilinear pairing 
$$
\langle -, -\rangle :\pi_n (G) \otimes \pi_m (G) \rightarrow \pi_{n+m} (G).
$$
called the \emph{Samelson product} \cite{Bott}. Let  $[f_P]\in \pi_{n-1}(G)$ be the clutching function of the principal bundle $P$ over $S^n$, then $s_P$ is identified with the Samelson product 
$$ s_P= \langle -,[f_P] \rangle:\pi_1(G) \rightarrow \pi_n (G)$$ 
see \cite[Theorem 2.7]{Lang} and \cite[Page 249]{Bott}.

Thus the computation of $\mathrm{aut}(\xi)$ is reduced to the computation of the corresponding Samelson product 
$$\langle - , - \rangle \colon \pi_1SO(m) \otimes \pi_{m-1}SO(m) \to \pi_mSO(m)$$ for $m=4,8$.  Let $SO(n) \to S^{n-1}$ be the evaluation map, 
$$p_n \colon \pi_{n-1}SO(n) \to \pi_{n-1}(S^{n-1})$$ 
be the induced homomorphism. 

\begin{Proposition}\label{prop:samelson}
Let $x \in \pi_{m-1}SO(m)$ such that $p_m(x) = \pm \mathrm{id}_{m-1}$, then the Samelson product 
$$ \langle - , x \rangle \colon \pi_1SO(m) \to \pi_mSO(m)$$
is non-trivial.
\end{Proposition}
\begin{proof}
For $p < q$, let $i \colon SO(p) \to SO(q)$ and $i' \colon SO(p) \to SO(q)$ be the inclusion induced by the inclusion of the first or last $p$ coordinates of the Euclidean spaces respectively, then we have a diagram commutative up to homotopy
$$\xymatrix{
SO(2) \wedge SO(m) \ar[r]^{i \wedge i' \ \ \ \ \ \ } \ar[d]^{i \wedge \mathrm{id}} & SO(m+1) \wedge SO(m+1) \ar[r]^{\ \ \ \ \ \ \ \ c} & SO(m+1) \\
SO(m) \wedge SO(m) \ar[r]^{ \ \ \ c} & SO(m) \ar[ur]^i & }$$
Let $\alpha_2\in \pi_1SO(2)$ be a generator, then $i_{*}(\alpha_2)=\alpha_m$ is the generator of $\pi_1SO(m)$. We will show $(c \circ (i \wedge i'))_*(\alpha_2 \otimes x)$ is non-trivial, then by the above diagram $\langle \alpha_m, x \rangle$ is non-trivial.

By \cite[Proposition 2.1 and 2.2]{Bott} the induced homomorphism in homotopy by $c \circ (i \wedge i')$ has the following form
$$(c \circ (i \wedge i'))_*=\Delta \circ \lambda^E_* \circ E \circ (p_2 \wedge p_m)$$
where $E$ is the suspension homomorphism, $\lambda^E \colon S^{m+1} \to S^{m+1}$ is a degree $1$ map, and $\Delta:\pi_{m+1}(S^{m+1})\rightarrow \pi_{m}SO(m+1)$ is the boundary homomorphism of the fibration $SO(m+1)\rightarrow SO(m+2)\rightarrow S^{m+1}$. 
Now since $p_1(\alpha_2)=\mathrm{id}_1$, $p_m(x)=\pm \mathrm{id}_{m-1}$, and the degree of $\lambda^E$ is $1$, we have 
$$ \lambda^E_*\circ E \circ (p_1 \wedge p_{m})(\alpha_2\otimes x)=\pm [\mathrm{id}_{m+1}]$$ 
It's well-known that the boundary homomorphism $\Delta$ is non-trivial for $m=4,8$ (\cite{Kervaire60}). Therefore $(c \circ (i \wedge i'))_*(\alpha_2 \otimes x)=\pm \Delta(\mathrm{id}_{m+1})$ is non-trivial.
\end{proof}

\begin{Lemma}\label{autxi}
When $m=4$, $\mathrm{aut}(\xi) \cong \z/2$; when $m=8$, $\mathrm{aut}(\xi) \cong \z/2 \oplus \z/2$.
\end{Lemma}
\begin{proof}
The exact sequence \ref{sequence2} reduces to
$$ \pi_1SO(m) \stackrel{s_P}{\rightarrow} \pi_{m}SO(m) \to \mathrm{aut}(\xi) \to 0.$$
we have $\pi_4SO(4) \cong (\z/2)^2$ and $\pi_8SO(8) \cong (\z/2)^3$ (c.f. \cite{Kervaire60}).
The Euler class of $\xi$ is $\pm 1$, which is equivalent to that, for the clutching function of $\xi$ in $\pi_{m-1}SO(m)$, its image in $\pi_{m-1}(S^{m-1})$ is $\pm \mathrm{id}$. By Proposition \ref{prop:samelson} $s_P$ is non-trivial. This finishes the proof.
\end{proof}

\begin{example}
Consider the automorphism $\tau$ of $\xi$ which is the antipodal involution on each fiber. It is shown \cite[Theorem 1]{TangZhang} that the induced diffeomorphism on $\partial D(\xi) = \partial N = S^7$ is isotopic to the antipodal map on $S^7$, hence is isotopic to the identity. Therefore from the isomorphisms $\mathrm{aut}(\xi) \cong \m(N) \cong I(M)$ we see $\tau=0 \in \mathrm{aut}(\xi)$.
\end{example}

\section{local diffeomorphims}\label{sec:local}
By the exact sequence (\ref{sequence}), the next step toward the determination of  $\m(M)$ is to study the homomorphism $\phi \colon \Theta_{2m+1} \to \m(M)$. To this end we first consider a variant of the mapping class group, namely the mapping class group relative to a disc.  Let $D \subset M$ be a closed embedded $2m$-disc, $\m(M, \mathrm{rel} D)$  be the group of isotopy classes of orientation-preserving diffeomorphism of $M$ which is the identity on $D$. There is an exact sequence relating $\m(M)$ and $\m(M, \mathrm{rel} D)$ (\cite[p.~265]{Wall2})
$$\z/2 \to \m(M,\mathrm{rel}D) \to \m(M) \to 0.$$
The image of $1 \in \z/2$ in $\m(M, \mathrm{rel}D)$ is a Dehn twist in a collar $A \cong S^{2m-1} \times [0,1]$ of $D$. More precisely, let $\alpha \colon [0,1] \to SO(m)$ be a smooth curve which maps a neighborhood of $\{0,1\}$ to the identity matrix, and represent the non-trivial element in $\pi_1SO(m) \cong \z/2$. Let $f(x,t)=(\alpha(t) \cdot x, t)$ be the diffeomorphism of $A$ and extend $f$ to a diffeomorphism of $M$ by the identity elsewhere.  We call this a \emph{boundary diffeomorphism}. 

By considering a diffeomorphism supported in an embedded $2m$-disc $D'$ disjoint from $D$ one clearly obtains a homomorphism  $\bar \phi \colon \Theta_{2m+1} \to \m(M,\mathrm{rel}D)$ lifting the homomorphism $\phi \colon \Theta_{2m+1} \to \m(M)$ 
$$\xymatrix{ & \m(M, \mathrm{rel}D) \ar[d]\\                    \Theta_{2m+1} \ar[ur]^{\bar\phi} \ar[r]^{\phi} & \m(M)}$$
In this section we study the homomorphism $\bar \phi$.

Let $X$ be the result of surgery on $S^{1}\times D \subset S^{1} \times M$, $I(X)$ be the inertia group of $X$. Then
\begin{Proposition}\label{prop:ker}
$\Ker  \bar \phi = I(X)$.
\end{Proposition}
\begin{proof}
For $f \in \m(D^{2m}, \partial)$, let $\Sigma_f \in \Theta_{2m+1}$ be the corresponding  homotopy sphere, then the mapping torus $M_f$ is diffeomorphic to $ S^1 \times M \sharp \Sigma_f$. Surgery on $S^1 \times D \subset S^1 \times M \sharp \Sigma_f \cong M_f$ produces $X \sharp \Sigma_f$. If $\Sigma_f \in \Ker  \bar \phi$, then $M_f$ is diffeomorphic to $S^1 \times M$ relative to $S^1 \times D$. After surgery on $S^1 \times D$ we have $X \sharp \Sigma_f \cong X$, i.~e.~$\Sigma_f \in I(X)$.

On the other hand, let $V$ be the trace of the surgery on $S^1 \times D \subset S^1 \times M$, $W=V \cup_{S^1 \times M} D^2 \times M$, with $\partial W = X$. Then a generator of $H_2(W) \cong \z$ is represented by an embedded $S^2$ with trivial normal bundle. Do a $2$-surgery on this $S^2$, killing $H_2(W) \cong \z$, still denote the result manifold by $W$. Note  that $X \sharp \Sigma_f = (N \times I)\cup_{f \cup \mathrm{id}}(N \times I)$ where $N = M - \mathring D$ (c.~f.~\cite[p.~657]{Kreck78}). If $X \sharp \Sigma_f \cong X$, an easy computation shows that $(W,N \times I)$ is an $h$-cobordism rel $\partial$. This implies that $f$ is pseudo-isotopic to $\mathrm{id}$ rel $D$, hence also isotopic to $\mathrm{id}$ rel $D$ by Cerf's pseudo-isotopy theorem \cite{Cerf}.
\end{proof}

We compute the inertia group $I(X)$ using the method of modified surgery theory. For basic concepts and results of this theory we refer to \cite{Kreck99} and \cite[\S 2]{KramerStolz}. Let $BO\langle m \rangle$ be the $(m-1)$-connected cover of $BO$,  or more precisely, $BO\langle 4 \rangle= B\mathrm{Spin}$, $BO\langle 8 \rangle=B\mathrm{String}$. Since $X$ is an $(m-1)$-connected $(2m+1)$-manifold with $H^m(X) = H^m(M)$ and $p_{m/4}(X)=p_{m/4}(M)$, there is a unique $BO\langle m \rangle $-structure on its stable normal bundle, and the normal $m$-type of $X$ is the same as that of $M$, which is described in  \cite[Proposition 3.4]{KramerStolz}: let $B_{2t+1,m}$ be the homotopy fiber of the map (unique up to homotopy) 
$BO\langle m \rangle \to K(\z,m)$, whose induced homomorphism $\pi_mBO\langle m \rangle = \z \to \pi_mK(\z ,m) = \z$ is multiplication by $2t+1$,
and let 
$$p \colon B_{2t+1,m} \to BO\langle m \rangle \to BO$$
be the composition, which we regard as a fibration. Then $(B_{2t+1,m},p)$ is the normal $m$-type of $X$, and 
there is a unique lifting $\bar \nu \colon X \to B_{2t+1,m}$ of the normal Gauss map $\nu \colon X \to BO$, which is an $(m+1)$-equivalence. To simplify the notation in the following we denote $B_{2t+1,m}$ by $B$.  Manifolds with a $B$-structure form a bordism group $\Omega_*(B,p)$ (c.~f.~\cite[\S 4]{KramerStolz}). The lifting $\bar \nu \colon X \to B$ gives rise to an element $[X, \bar \nu] \in \Omega_{2m+1}(B,p)$. 
A homotopy sphere $\Sigma$ also has a unique normal $B$-structure $\bar \nu_{\Sigma} \colon \Sigma \to B$ and we have a homomorphism $\Theta_{2m+1} \to \Omega_{2m+1}(B,p)$.

\begin{Proposition}\label{prop:inertia}
$I(X)=\Ker(\Theta_{2m+1} \to \Omega_{2m+1}(B,p))$
\end{Proposition}
\begin{proof}
 First note that we have the identity 
$[X, \bar \nu] + [\Sigma, \bar \nu_{\Sigma}]=[X \sharp \Sigma, \bar \nu']$ where $\bar \nu'$ is the unique normal $B$-structure on $X \sharp \Sigma$. If $X \sharp \Sigma \cong X$, since the normal $B$-structure $\bar \nu$ on $X$ is unique up to homotopy, we have  $[\Sigma, \bar \nu_{\Sigma}] =0 \in \Omega_{2m+1}(B,p)$. This shows 
$$I(X) \subset \Ker (\Theta_{2m+1} \to \Omega_{2m+1}(B,p)).$$

To prove the other direction of the inclusion we define a homomorphism 
$$\theta \colon\Omega_{2m+2}(B,p) \to L_{2m+2}(\z)$$
where $L_{2m+2}(\z) \cong \z/2$ is Wall's group of simply-connected surgery obstruction. For an element $[N,g] \in \Omega_{2m+2}(B,p)$, consider the surgery problem
$$\bar \nu + g \colon X \times I + N \to B,$$ 
Since we are in the case $[(2m+1)/2]=m$, the modified surgery obstruction $\theta$ is in the group $L_{2m+2}(\z)$ (\cite[p.~708]{Kreck99}). It's easy to verify by definition that $\theta$ is a homomorphism.

Brown \cite{Brown} defined a generalization of the Kervaire invariant 
$$\Psi \colon \Omega_{2m+2}^{\mathrm{Spin}}  \to \z/2.$$ 
For $m=8$, we still denote by $\Psi$ the composition $\Omega^{\mathrm{string}}_{18} \to \Omega^{\mathrm{spin}}_{18} \to \z/2$. It's shown in \cite[Theorem 1]{Brown1} that for $\alpha \in \pi_1^S$ the non-trivial element and $\beta \in \Omega^{\mathrm{spin}}_{8k}$
$$\Psi(\alpha^2 \cdot \beta) \equiv \chi(\beta) \pmod 2$$ 
where $\chi(\beta)$ is the Euler characteristic. Taking $\beta= \mathbb H \mathrm P^2$ or $\mathbb O \mathrm P^2$, we see that $\Psi \colon \Omega^{\mathrm{spin}}_{10} \to \z/2$ and $\Psi \colon \Omega^{\mathrm{string}}_{18}  \to \z/2$ are surjective. The map 
$B \to BO\langle m \rangle $ induces a homomorphism of bordism groups $\Omega_{2m+2}(B,p) \to
                     \Omega_{2m+2}^{BO\langle m \rangle} $.

\begin{Lemma}
There is a commutative diagram
$$\xymatrix{\Omega_{2m+2}(B,p) \ar[d] \ar[r]^{\ \ \ \theta} & \z/2 \\
                     \Omega_{2m+2}^{BO\langle m \rangle} \ar[ur]^{\Psi} & }$$

\end{Lemma}
\begin{proof}
For $[N,g] \in \Omega_{2m+1}(B,p)$, we may first do surgery to make $g$ being $(m-1)$-connected and $g_* \colon \pi_m(N) \to \pi_m(B)$ an isomorphism. Let $S^m_N \subset N$, $S^m_X \subset X$ be embedded spheres representing a generator of $\pi_m(N)=H_m(N) \cong \z$ and $\pi_m(X)=H_m(X)\cong \z$ respectively. Let $\{e_i\}$ be a symplectic basis of $H_{m+1}(N)$, choosing embedded spheres $S^{m+1}_i \subset N$ representing $e_i$. When $m=4$, the normal bundle of $S^5_i$ is  stably trivial since $\pi_5(BSO)=0$. When $m=8$, note that $\pi_9(B) =\pi_9(BSO) \cong \z/2$, a generator is given by $S^9 \stackrel{\eta}{\rightarrow} S^8 \stackrel{f}{\rightarrow} B$, where $\eta$ is (the suspension of) the Hopf map and $f$ is a generator of $\pi_8(B)=\pi_8(BSO) \cong \z$. Therefore there is a null-homologous embedded sphere $S^9_0 \subset N$ with non-trivial stable normal bundle. Now if the stable normal bundle of $S^9_i \subset N$ is non-trivial, we may take a connected-sum of $S^9_i$ with the above $S^9_0$, this will not change the homology class of $S^9_i$ and the stable normal bundle of the new embedded sphere is trivial. Brown \cite{Brown} defined a quadratic form $\psi_N \colon H^{m+1}(N;\z/2) \to H^{2m+2}(N;\z/2)=\z/2$ and showed that the normal bundle of $S^{m+1}_i$ is trivial if and only if $\psi_N(De_i)=0$, where $D$ is the Poincar\'e dual map. Let $W=X \sharp N$, then the normal bundle of $S^m_X \sharp S^m_N \subset W$ is trivial and we may do surgery on $S^m_X \sharp S^m_N$ to make $(W,X)$ $m$-connected. This will not produce new $(m+1)$-dimensional homology classes since $S^m_N$ is dual to an $(m+2)$-cycle in $N$.  Now $\{S^{m+1}_i\}$ represent a symplectic basis of the surgery kernel $H_{m+1}(W,X)$, therefore the Arf-invariant of surgery obstruction coincides with the Arf-invariant of $\psi_N$, which is $\Psi(N)$.
\end{proof}

Now we continue the proof of Proposition \ref{prop:inertia}. The homomorphism $\Omega_{2m+2}(B,p) \to \Omega_{2m+2}^{BO\langle m \rangle}$ is a $2$-local isomorphism (\cite[Lemma 4.2]{KramerStolz}), meaning that the kernel and the cokernel are finite groups of odd order, and $\Psi \colon \Omega_{2m+2}^{BO\langle m \rangle}  \to \z/2$ 
is surjective, therefore $\theta \colon \Omega_{2m+2}(B,p) \to \z/2$ is surjective. For $\Sigma \in \Ker(\Theta_{2m+1} \to \Omega_{2m+1}(B,p))$,  $X$ and $X \sharp \Sigma$ are $B$-bordant. Let $W$ be a normal $B$-bordism between $X \sharp \Sigma$ and $X$, we do surgery on $W$ to make $(W, X)$ an $h$-cobordism. If the surgery obstruction is non-trivial, then by concatenating a $B$-bordism between $X$ and $X$ with non-trivial surgery obstruction, we have a $B$-bordism between $X\sharp \Sigma$ and $X$ with trivial surgery obstruction. This shows that $\Sigma \in I(X)$.
\end{proof}

By Propositions \ref{prop:ker} and \ref{prop:inertia} we have identifications
\begin{eqnarray*}
\Ker \bar \phi = I(X) &  = & \Ker(\Theta_{2m+1} \to \Omega_{2m+1}(B,p)) \\
& = & \Ker(\Theta_{2m+1} \to \Omega_{2m+1}^{BO\langle m \rangle})
\end{eqnarray*}
where the last equality comes from the facts that the homomorphism $\Omega_{*}(B,p) \to \Omega_{*}^{BO\langle m \rangle}$ is a $2$-local isomorphism and  $\Theta_{9} \cong (\z/2)^{3}$,  $\Theta_{17} \cong (\z/2)^4$.
Knowledge of the homomorphism $\Theta_{2m+1} \to \Omega_{2m+1}^{BO\langle m \rangle}$ can be mostly found in the literature, though some calculations are needed for $m=8$.

When $m=4$,  $\Omega_{2m+1}^{BO\langle m \rangle} = \Omega_9^{\mathrm{spin}}$,  it's well-known that the image of $\Theta_9 \to \Omega_9^{\mathrm{spin}}$ is isomorphic to $\z/2$ and detected by the $\alpha$-invariant $\hat{\mathscr A} \colon \Omega_9^{\mathrm{spin}} \to  KO^{-9}$ (see \cite[p.~350, Theorem]{Stong} and \cite[\S 3]{Milnor65}).

When $m=8$, we have $\Omega_{2m+1}^{BO\langle m \rangle} = \Omega_{17}^{\mathrm{string}} \cong (\z/2)^3$ (see \cite[Theorem 1.2]{Davis} and \cite[Theorem 4.3, Theorem 4.5]{GiamString}), $\pi_{17}^S \cong (\z/2)^4$ with generators $\eta  \eta^*$, $\eta^2  \rho$,   $\nu  \kappa$ and $\bar{\mu}$, where $\eta$ is the generator of $\pi_1^S \cong \z/2$, $\eta^* \in \pi_{16}^S$ is an element of order $2$, $\rho \in \pi_{15}^S$ is an element of order $32$ (\cite[Chap XIV]{Toda}). It's also known that  the image of the $J$-homomorphism  $J \colon \pi_{17}(O) \to \pi_{17}^S$ is generated by $\eta^2  \rho$.

\begin{Lemma}\label{kerHurewicz}
The kernel of the Hurewicz homomorphism $$h^{\mathrm{string}}_{17}:\pi_{17}^S \rightarrow \Omega_{17}^{\mathrm{string}}$$ 
is isomorphic to $\z/2\oplus \z/2$, generated by $\eta \eta^*$ and $\eta^2 \rho$; the image of $h_{17}^{\mathrm{string}}$ is isomorphic to $\z/2\oplus \z/2$, generated by $\bar \mu$ and $\nu \kappa$. $\Omega_{17}^{\mathrm{string}} \cong (\z/2)^3$ is generated by $\bar \mu$, $\nu \kappa$ and $[(S^1,\tau) \times M]$, where $\tau$ stands for the non-trivial spin structure on $S^{1}$. 
\end{Lemma}

\begin{proof}
The Hurewicz homomorphism $h_{*}^{\mathrm{string}} \colon \pi_{*}^S \to \Omega_{*}^{\mathrm{string}}$ is a ring homomorphism and $h_{16}^{\mathrm{string}}$ is trivial since $\Omega_{16}^{\mathrm{string}}$ is isomorphic to $\z\oplus \z$ (\cite[p.~538]{GiamString}) and $\pi_{16}^{S}$ is finite. Therefore
$$h^{\mathrm{string}}_{17}(\eta\eta^*)=h^{\mathrm{string}}_{16}(\eta^*)h^{\mathrm{string}}_{1}(\eta)=0$$
and $$h^{\mathrm{string}}_{17}(\eta^2\rho)=h^{\mathrm{string}}_{16}(\eta\rho)h^{\mathrm{string}}_{1}(\eta)=0.$$

It is known that the periodic element $\bar{\mu}$ has non-trivial $\alpha$-invariant (since  $\bar \mu \in \langle \mu, 2 \iota, 8 \sigma \rangle $ by \cite[Chap XIV]{Toda}, the $\alpha$-invariant of elements in $\pi_*^S$ coincides with Adams' invariant $d_{\mathbb R}$ and $d_{\mathbb R}(\bar \mu) =1$ by \cite[Example 12.19]{Adams}). Thus $h^{\mathrm{string}}_{17}(\bar{\mu})$ is non-trivial. 

The element $\nu\kappa$ is represented by  $h_2d_0=h_0e_0$ in the $E_2$-page of the Adams spectral sequence for $\pi_*^S$, $h_0e_0$ is mapped nontrivially to the $E_2$-page of the Adams spectral sequence for $\pi_*(M\mathrm{String})$ and this element survives to infinity (see \cite[Table 4.34]{Tangora70}, \cite[Theorem 3.3]{GiamString}, \cite[Theorem 2.2]{Davis} and \cite[Appendix A3]{Ravenel}). Therefore $h^{\mathrm{string}}_{17}(\nu \kappa)$ is non-trivial. But the  $\alpha$-invariant of  $\nu\kappa$ is zero since $KO^{-3}(\mathrm{pt})=KO^{-6}(\mathrm{pt})=0$. It follows that the images of $\nu\kappa$ and $\bar{
\mu}$ are linearly independent in $\Omega^{\mathrm{string}}_{17}$.

That $[(S^1, \tau) \times M]$ generates another copy of $\z/2$ in $\Omega^{\mathrm{string}}_{17}$ follows from Lemma  \ref{lem:bound}.
\end{proof}

There is a short exact sequence
$$0 \to bP_{18} \to \Theta_{17} \to \pi_{17}^S/\mathrm{Im}J \to 0$$
where $bP_{18}$ consists of homotopy spheres which bound parallelizable manifolds. Clearly elements in $bP_{18}$ are mapped to $0$ in $\Omega_{17}^{\mathrm{string}}$. Summarizing the above information, we have the following lemma.

\begin{Lemma}\label{lem:comp}
The image of  $\Theta_{9} \to \Omega_{9}^{\mathrm{spin}}$ is isomorphic to $\z/2$, generated by a Hitchin sphere; the image of $\Theta_{17} \to \Omega_{17}^{\mathrm{string}}$ is isomorphic to $\z/2\oplus \z/2$, generated by $\bar \mu$ and $\nu \kappa$.
\end{Lemma}

\section{Proof of the theorems}\label{sec:proof}

First we show that the boundary diffeomorphism is not isotopic  rel $D$ to a local diffeomorphism. This will imply $\ker \phi=\ker \bar \phi$. Do a surgery on the embedded $S^1 \times D$ in the mapping torus of the boundary diffeomorphism, denote the result manifold by $X'$. Let $\tau$ be the non-trivial spin structure on $S^1$,  then $X'$ is the result of the spin surgery on $S^1 \times D \subset (S^1, \tau )\times M $. If the boundary diffeomorphism were isotopic rel $D$ to a local diffeomorphism, then $X'$ would be diffeomorphic to $X \sharp \Sigma$ for a Hitchin sphere $\Sigma \in \Theta_{2m+1}$. Since $S^1 \times M$ with the trivial spin structure is null-bordant in $\Omega_{2m+1}^{\mathrm{spin}}$, this would imply $[(S^1, \tau) \times M]=[\Sigma] \in \Omega_{2m+1}^{\mathrm{spin}}$. But this is ruled out by the following lemma.

\begin{Lemma}\label{lem:bound}
$[(S^1, \tau) \times M]$ is not in the image of $\Theta_{2m+1} \to \Omega_{2m+1}^{\mathrm{spin}}$.
\end{Lemma}
\begin{proof}
This will be proved by computation in the corresponding bordism group. It is known that the image of $\Theta_{2m+1}$ in $\Omega^{\mathrm{spin}}_{2m+1}$ is isomorphsic to $\z/2$, detected by the $\alpha$-invariant (\cite[p.~350, Theorem]{Stong}). There is an $8$-dimensional almost-parallelizable manifold $Y$ with $\hat A(Y)$ $=1$ and $\mathrm{sign}(Y)=-7\cdot 2^5$ (\cite[\S 3]{Milnor65}).

When $m=4$, the group $\Omega_8^{\mathrm{spin}} $ is isomorphic to $\z \oplus \z$, generated by $\mathbb H \mathrm P^2$ and $Y$.  In this group we have $[M]= a[\mathbb H \mathrm P^2] + b [Y]$ with $a$ odd since $p_1^2[\mathbb H \mathrm P^2]=4$, $p_1^2[Y]=0$ and $p_1^2[M]=4(1+2t)^2$. The group $\Omega_9^{\mathrm{spin}}$ is isomorphic to $\z/2 \oplus \z/2$, generated by $[(S^1, \tau)\times \mathbb H \mathrm P^2]$ and $[(S^1, \tau) \times  Y]$ (\cite[p.~258]{ABP66}).  The non-trivial element in the image of $\Theta_9 \to \Omega_9^{\mathrm{spin}}$ equals $[(S^1, \tau) \times Y]$ (\cite[\S 3]{Milnor65}). This shows  $[(S^1, \tau) \times M] \ne [\Sigma]$ in  $\Omega_9^{\mathrm{spin}}$. 

When $m=8$, the group $\Omega_{16}^{\mathrm{spin}}\otimes \z/2\cong (\z/2)^5$ is generated by $\hp^4$, $\hp^2\times \hp^2$, $\hp^2 \times Y$, $Y \times Y$ and $N \times K$, where $N$ is a 12-dimensional spin-manifold,  $K$ is the Kummer surface, a generator of $\Omega_{4}^{\mathrm{spin}}\cong \z$ (\cite[Page 258]{ABP66}). In this group we have 
$$
[M] = a [\hp^2\times \hp^2]+b[\hp^4]+c[\hp^2 \times Y] +d [Y \times Y] +e [N \times K],
$$
with $a+b=1$ since the signature of $M$ is equal to 1. Taking product with $S^1$ with the non-trivial spin structure gives an isomorphism 
$$ - \times [S^1, \tau] \colon \Omega_{16}^{\mathrm{spin}}\otimes \z/2  \stackrel{\cong}{\longrightarrow} \Omega_{17}^{\mathrm{spin}}\otimes \z/2$$
The non-trivial element in the image of $\Theta_{17} \to \Omega_{17}^{\mathrm{spin}}$ equals $[(S^1, \tau) \times Y \times Y]$ (\cite{Milnor65}). This shows $[(S^1, \tau) \times M] \ne [\Sigma]$.
\end{proof}

Now we are ready to finish the proof of the main theorems. We have determined $\Ker \bar \phi$ in Lemma \ref{lem:comp}, and Lemma \ref{lem:bound} implies $\Ker \phi = \Ker \bar \phi$. 

When $m=4$, in the exact sequence (\ref{sequence}) we see $\m(N) \cong \mathrm{aut}(\xi) \cong \z/2$ and $\gamma \colon \m(N) \to I(M)$ is an isomorphism. Therefore $\Theta_9 \to \m(M)$ is surjective, the image is isomorphic to $\z/2$, generated by a Hitchin sphere. If the $\hat A$-genus of $M$ is even, then $\hat{\mathscr A}(S^1 \times M)=0$, no matter which spin structure on $S^1 \times M$ is taken. Therefore for any local diffeomorphism $f$ with corresponding homotopy sphere $\Sigma_f$, 
$$\hat{\mathscr A}(M_f)=\hat{\mathscr A}(S^1 \times M \sharp \Sigma_f)=\hat{\mathscr A}(\Sigma_f).$$ 
This proves Theorem \ref{thm:main1}.

When $m=8$, $\m(M) \cong \mathrm{aut}(\xi) \cong (\z/2)^2$, and $\gamma \colon \m(M) \to I(M)$ is surjective. The exact sequence (\ref{sequence}) reduces to 
\begin{equation}\label{seq2}
 0 \to \Theta_{17}/\Ker \phi \to \m(M) \to \z/2 \to 0
 \end{equation}
with $\Theta_{17}/\ker \phi \cong \z/2 \oplus \z/2$, generated by $\overline \mu$ and $\nu \kappa$. This is a central extension since local diffeomorphisms commute with any diffeomorphism by the Disc Theorem. The proof of Theorem \ref{thm:main2} is completed by the following lemma.

\begin{Lemma}
Elements in $\m(M)$ have order $2$.
\end{Lemma}
\begin{proof}
The mapping torus construction gives a homomorphism (c.~f.~\cite[Page 656]{Kreck78})
$$
\mu \colon \m(M, \mathrm{rel}D) \to \Omega_{17}^{\mathrm{string}}, \ \ f \to [M_f].$$
This induces a homomorphism $\bar \mu \colon \m(M) \to \Omega_{17}^{\mathrm{string}}/H$, where $H$ is the subgroup generated by $[(S^1, \tau) \times M]$.  For any $f \in \m(M)$, by the exact sequence (\ref{seq2})  $f^2$ is isotopic to a local diffeomorphism $\varphi$. Since $\Omega_{17}^{\mathrm{string}}/H \cong \z/2 \oplus \z/2$, $\bar \mu(\varphi) = \bar \mu (f^2) = 2 \cdot \bar \mu(f)=0$. But $\bar \mu (\varphi)=[\Sigma_{\varphi}]$, therefore $\Sigma_{\varphi} \in \Ker (\Theta_{17} \to \Omega_{17}^{\mathrm{string}}/H)$. By Lemma \ref{kerHurewicz} 
$$\Ker (\Theta_{17} \to \Omega_{17}^{\mathrm{string}}/H) = \Ker (\Theta_{17} \to \Omega_{17}^{\mathrm{string}}).$$ 
This shows $\varphi =0 \in \m(M)$.
\end{proof}

\begin{proof}[Proof of Corollary \ref{cor:curvature}]
Fix a metric $g$ of positive scalar curvature, the action of $\mathrm{Diff}(\mathbb H \mathrm P^2)$ on $\mathcal{R}^+_{\mathrm{scal}}$ gives a map $\mathrm{Diff}(\mathbb H \mathrm P^2) \to \mathcal{R}^+_{\mathrm{scal}}$, $h \mapsto h^*g$, which induces a map $\m(\mathbb H \mathrm P^2) \to \pi_0 \mathcal{R}^+_{\mathrm{scal}}$. It's shown in  \cite{Hitchin} that the map $\m(\mathbb H \mathrm P^2) \to  KO^{-9}(\mathrm{pt})$ given by the $\alpha$-invariant of the mapping torus has a factorization
$$\m(\mathbb H \mathrm P^2) \to \pi_0 \mathcal{R}^+_{\mathrm{scal}} \to KO^{-9}(\mathrm{pt}).$$
Theorem \ref{thm:main1} implies that for the non-trivial element $h \in \m(\mathbb H \mathrm P^2)$, $h^*g$ and $g$ are not in the same path component of $\mathcal{R}^+_{\mathrm{scal}}$. Hence the map $\pi_0\mathcal{R}^+_{\mathrm{scal}} \to \pi_0 \mathscr M^+_{\mathrm{scal}}$ is two-to-one.

Let $f$ be such an isometry, it induces an isometric action of $\z$ on $\mathbb R \times \mathbb H \mathrm P^2$ (with the product metric) by $(t,x) \mapsto (t+1, f(x))$. The orbit manifold is the mapping torus $M_f$, which clearly inherits a Riemannian metric of positive scalar curvature, hence $\hat{\mathscr{A}}(M_f)=0$. 
\end{proof}


\begin{thebibliography}{50}
	\bibitem{Adams}J. F. Adams {\em On the groups $J(X)$ IV}. Topology 5, 21-71  (1966)
	\bibitem{ABP66}D. W. Anderson, E. H. Brown and F. P. Peterson, {\em Spin cobordism}, Bull. Amer. Math. Soc. 72, 256–260 (1966)
    \bibitem{ABPSpin67} D.W. Anderson, E.M Jr. Brown and F.P. Peterson, {\em The structure of the $Spin$ cobordism ring.} Ann. of Math. 86 (1967), 271–298. 
    \bibitem{AtiyahBott}M. Atiyah, R. Bott  {\em The Yang-Mills equations over Riemann surfaces}. Philos Trans R Soc Lond Ser A Math Phys Eng
Sci, 1983, 308: 523–615
    \bibitem{Baues}H.J. Baues, {\em On the group of homotopy equivalences of a manifold}, Trans. Amer. Math. Soc., 348, No. 12, 4737-4773 (1996)
    \bibitem{Bott}R. Bott, {\em A note on the Samelson product in the classical groups}, Comment. Math. Helv. 34 (1960),
249-256.

  \bibitem{Brown} E.~Brown, 
{\em Generalizations of the Kervaire invariant}, Ann. Math. 95 (1972), 368-383.  
  \bibitem{Brown1} E.~Brown, F. P.~Peterson, {\em The Kervaire invariant of $(8k+2)$-manifolds}, Bull. Amer. Math. Soc.71 (1965), 190–193.
    \bibitem{Cerf}J. Cerf, {\em The pseudo-isotopy theorem for simply connected differentiable manifolds},
Manifolds-Amsterdam (1970), 76-82, Lecture Notes in Math. 197, Springer, Berlin (1970),
\bibitem{Davis}D.~Davis, M.~Mahowald, {\em A new spectrum related to
7
-connected cobordism}, Algebraic Topology (Arcata, CA, 1986), Lecture Notes in Math., vol. 1370, Springer-Verlag, Berlin-New York, 1989, pp. 126–134.
    \bibitem{EellsKuiper62} J. Eells, Jr. \& N.H. Kuiper, {\em Manifolds which are like projective planes}, Inst.Hautes \'{E}tudes Sci. Publ. Math. 14(1962) 5-46,
    
	\bibitem{GiamString} V. Giambalvo, {\em On $\langle 8 \rangle$-cobordism}, Illinois J. Math. 15 (1971) 533–541.
	\bibitem{Gottlieb}D.H. Gottlieb, {\em Applications of bundle map theory}. Trans Amer Math Soc, 1972, 171: 23–50
	\bibitem{Haef}A. Haefliger, {\em Differentiable Imbeddings}, Bull. Amer. Math. Soc. 67 (1961), 109-112 
	\bibitem{Hitchin}N. Hitchin, {\em Harmonic spinors}, Adv. Math. 14 (1974) 1–55. 
	\bibitem{Kervaire60} M.A. Kervaire, {\em Some nonstable homotopy groups of Lie groups}, Illinois J. Math. 4 (1960), 161-169.
	\bibitem{Kramer03} L. Kramer, {\em Projective planes and their look-alikes}, J. Differential Geom. 64(1) (2003), 1–55, 	
	\bibitem{KramerStolz} L.Kramer, S. Stolz, {\em A diffeomorphism classification of manifolds which are like projective planes.} J.Differential Geometry. 77.(2007), 177-188
	\bibitem{Kreck78}M. Kreck, {\em Isotopy classes of diffeomorphisms of $(k-1)$-connected almost parallelizable $2k$-manifolds}, Algebraic topology, Aarhus 1978 (Proc. Sympos., Univ. Aarhus, Aarhus, 1978),
pp. 643-663, Lecture Notes in Math., 763, Springer, Berlin (1979).

\bibitem{Kreck99} M.~Kreck, {\em Surgery and duality}, Ann. Math. 149 (1999), 707–754.

\bibitem{Lang} G.E.~Lang, {\em The evaluation and EHP sequences},  Pacific J. Math. 44 (1973), 201-210.
\bibitem{Milnor65}J.~Milnor, {\em Remarks concerning spin manifolds}. 1965 Differential and Combinatorial Topology (A Symposium in Honor of Marston Morse) pp. 55–62 Princeton Univ. Press, Princeton, N.J.
\bibitem{Ravenel} D. C. Ravenel, {\em Complex Cobordism and Stable Homotopy Groups of Spheres}, Pure and Applied Mathematics, Academic Press, Inc., Orlando, FL, 1986, vol. 121.

\bibitem{Shimata} N.~Shimata, {\em Differential structures on certain 15-dimensional manifolds and Pontrjagin classes}, Nagoya Math. J. 12(1957), 59-69. 
\bibitem{Stong} R.~Stong, {\em Note on cobordism theory}, Mathematical notes Princeton University
Press, Princeton, N.J.; University of Tokyo Press, Tokyo 1968

\bibitem{TangZhang}Z. Tang and W.P.  Zhang, {\em $\eta$-Invariant and a  problem of B\'{e}rard-Bergery on the  existence of closed geodesics}, Adv. Math, 254(2014), 41-48.

\bibitem{Tangora70} M.~C.~Tangora, {\em On the cohomology of the Steenrod algebra}, Math. Z. 116 (1970), 18–64.

\bibitem{Theriault}S.D. Theriault, {\em Odd primary homotopy decompositions of gauge groups}. Algebr Geom Topol, 2010, 10: 535–564
	\bibitem{Toda}H. Toda, {\em Composition methods in homotopy groups of spheres}, Ann. of Math. Studies No. 49, Princeton Univ. Press, Princeton, N. J., 1962
	\bibitem{TuschmannWraith}W. Tuschmann, D. Wraith, {\em Moduli Spaces of Riemannian Metrics}, Second corrected printing, volume 46 of Oberwolfach Seminars. Birkh\"auser, Basel(2015)
	\bibitem{Wall2} C.T.C.~Wall, \emph{ Classification problems in differential topology II}, Topology. 2 (1963) 263–272. 

\end{thebibliography}
\end{document}